\documentclass[10pt]{amsart}

\usepackage{amsmath, amscd, amssymb}
\usepackage[frame,cmtip,arrow,matrix,line,graph,curve]{xy}
\usepackage{graphpap, color}
\usepackage[mathscr]{eucal}
\usepackage{cancel}
\usepackage{verbatim}

\numberwithin{equation}{section}

\def\sO{{\mathscr O}}

\newcommand{\CC}{\mathbb{C}}

\newcommand{\PP}{\mathbb{P}}

\newcommand{\ZZ}{\mathbb{Z}}

\def\cO{{\cal O}}

\newcommand{\tE}{\tilde{E}}

\def\and{\quad{\rm and}\quad}

\def\mapright#1{\,\smash{\mathop{\lra}\limits^{#1}}\,}

\newtheorem{prop}{Proposition}[section]
\newtheorem{theo}[prop]{Theorem}
\newtheorem{lemm}[prop]{Lemma}

\newtheorem{rema}[prop]{Remark}
\newtheorem{exam}[prop]{Example}
\newtheorem{ques}[prop]{Question}
\newtheorem{defi}[prop]{Definition}
\newtheorem{conj}[prop]{Conjecture}
\newtheorem{prob}[prop]{Problem}

\def\beq{\begin{equation}}
\def\eeq{\end{equation}}

\def\PP{\mathbb{P}}
\def\CC{\mathbb{C}}

\def\lra{\longrightarrow}
\def\mapright#1{\,\smash{\mathop{\lra}\limits^{#1}}\,}
\def\cO{\mathcal{O}}

\def\fX{\mathfrak X}

\def\Pic{\mathrm{Pic}}

\def\red{\color{red}}
\def\black{\color{black}}

\title{All complete intersection varieties are Fano visitors}

\author{Young-Hoon Kiem, In-kyun Kim, Hwayoung Lee, and Kyoung-Seog Lee}

\address{Department of Mathematics and Research Institute
of Mathematics, Seoul National University, Seoul 151-747, Korea}
\email{kiem@math.snu.ac.kr}

\address{Department of Mathematics, Seoul National University, Seoul 151-747, Korea}
\email{soulcraw@gmail.com}

\address{Department of Mathematics, Seoul National University, Seoul 151-747, Korea}
\email{hlee014@snu.ac.kr}

\address{School of Mathematics, Korea Institute for Advanced Study, Seoul 130-722, Korea}
\email{kyoungseog02@gmail.com}

\thanks{Partially supported by NRF grant 2011-0027969.}

\begin{document}
\begin{abstract}
We prove that the derived category of a smooth complete intersection variety is equivalent to a full subcategory of the derived category of a smooth projective Fano variety. This enables us to define some new invariants of smooth projective varieties and raise many interesting questions. 
\end{abstract}
\maketitle

\section{Introduction}\label{s1}

%\red  Fano visitors for complete intersection Calabi-Yau varieties, Bondal's question \black  

%Fano varieties are one of the most important classes of algebraic varieties 
%play a very special role in birational geometry. 

In recent years, derived categories of projective varieties attracted tremendous interest among algebraic geometers as well as physicists. Fano varieties in particular have been most intensively studied because their derived categories (1) determine the varieties completely due to the reconstruction theorem of Bondal and Orlov and (2) have canonical semiorthogonal  decompositions by canonical exceptional collections of line bundles by the Kodaira vanishing theorem. Moreover, the derived categories of many Fano varieties of low dimension have been calculated quite explicitly.

Often the derived categories of Fano varieties are big enough to contain interesting subcategories. For example, the derived categories of hyperelliptic curves are full subcategories of the derived categories of intersections of two quadrics (cf. \cite{BO})
and the derived categories of some special cubic 4-folds contain the derived categories of $K3$ surfaces (cf. \cite{Kuz1}).
These results provide derived category theoretic explanations for the corresponding geometric results of \cite{DR, Reid} and \cite{BD}. Analysis of derived categories may tell us where to dig (or where not to dig) when we are searching for a specific type of varieties. 

In 2011, Bondal raised the following question (cf. \cite{BBF}).
\begin{ques}\label{q1.0} \emph{(Fano visitor problem)}\\
Let $Y$ be a smooth projective variety. Is there any smooth projective Fano variety $X$ together with a fully faithful embedding $D^b(Y)\to D^b(X)$? If such an $X$ exists, $Y$ is called a \emph{Fano visitor}  and $X$ is called a \emph{Fano host} of $Y$.  
\end{ques}

If the answer to this question is yes for all smooth projective varieties, namely if all smooth projective varieties are Fano visitors, then we can effectively reduce the study of the derived categories of all smooth projective varieties to the Fano case. Moreover we can define interesting new invariants of projective varieties $Y$ by considering for instance the minimal dimension of Fano hosts of $Y$ (cf. \S\ref{s5}).

Here is the current state of knowledge for Bondal's question: 
In \cite{BBF}, Bernardara, Bolognesi and Faenzi use homological projective duality to find a list of Fano visitors, including all smooth plane curves. In \cite{SeTh}, Segal and Thomas prove that for a general quintic 3-fold $Y$, there is a Fano 11-fold $X$ such that there is a fully faithful embedding $D^b(Y) \hookrightarrow D^b(X).$ 

The purpose of this paper is to prove the following.
\begin{theo}\label{t1.1} %(Theorem \ref{t4.3})
All complete intersection smooth projective varieties are Fano visitors.
\end{theo}
Our construction is quite elementary and completely different from \cite{BBF, SeTh} even in the case of plane curves or quintic 3-folds. 
Our Fano host for a quintic 3-fold is only 5 dimensional. 
%For instance, $S$ is a Grassmannian and $E$ is the direct sum of positive line bundles 
We expect that our construction will work for larger classes of varieties. 

A lot of interesting questions follow from Theorem \ref{t1.1}. For instance, we can ask for minimal dimensional choices of Fano hosts.
\begin{defi}\label{d1.2} %(Definition \ref{d5.1})
The \emph{Fano dimension} of a smooth projective variety $Y$ is the minimum dimension of Fano hosts $X$ of $Y$.
%of $\dim X$ for smooth projective 
%Fano varieties $X$ which admit fully faithful embeddings $D^b(Y)\hookrightarrow D^b(X)$. 
We define the Fano dimension to be infinite if no Fano hosts exist.

The \emph{Fano number} of a smooth projective variety $Y$ is the number of deformation equivalence classes of (irreducible) Fano hosts $X$ of $Y$ of minimal dimension (i.e. $\dim X$ equals the Fano dimension of $Y$).
\end{defi}
\begin{ques}\label{q1.2}
What is the Fano dimension of a complete intersection variety? What is the Fano number of each of these varieties?
\end{ques}

We can also refine Bondal's question.
\begin{ques} \label{q1.3}%(Definition \ref{q5.2})
Let $Y$ be a smooth projective variety. 
Are there Fano hosts $X$ of $Y$ with prescribed Hodge numbers?
\end{ques}
For instance, when $Y$ is holomorphic symplectic, we may ask for a Fano host $X$ satisfying $h^{3,1}(X)=1$.
Another possible way to refine Bondal's question is to require $X$ to be of certain simple type.
\begin{ques} \label{q1.4}
Let $Y$ be a smooth projective variety. 
Is there a Fano host $X$ of $Y$, which is a hypersurface in a toric projective variety?
\end{ques}
Our construction in \S\ref{s4} gives $X$ as a hypersurface in a projective bundle over a projective space. Hence the answer to Question \ref{q1.4} is yes for all complete intersection smooth projective varieties. 

We can also generalize the notion of Fano number.
\begin{defi}\label{d1.5}
Let $Y$ be a smooth projective variety. 
Let $\kappa_Y(n)$ be the number of deformation equivalence classes of (irreducible)  Fano hosts $X$ of dimension $n$.
Let $\tau_Y(n)$ be the number of deformation equivalence classes of (irreducible) Fano hosts $X$ of dimension $n$, which is a hypersurface in a toric projective variety.
These define functions $\kappa_Y:\ZZ_{>0}\to \ZZ_{\ge 0}$ and $\tau_Y:\ZZ_{>0}\to \ZZ_{\ge 0}$, which we call the \emph{Fano function} and \emph{toric Fano function} respectively.  
\end{defi}\
Note that toric Fano varieties are well understood at least in low dimension.
%, it may be possible to calculate the toric Fano function for some simple examples. 

It is obvious that the Fano function determines the variety $Y$ when $Y$ is $\PP^1$. 
It will be interesting to investigate how strong the  (toric) Fano function is as an invariant of the variety $Y$. For instance, we conjecture that the (toric) Fano function determines del Pezzo surfaces.

\medskip

\noindent \textbf{Acknowledgement}. We thank Atanas Iliev, Dmitri Orlov, Genki Ouchi and Richard Thomas for useful comments.

\medskip
\noindent\textbf{Notation}. In this paper, all varieties are defined over the complex number field $\CC$. For a vector bundle $F$ on $S$, the projectivization $\PP F:=\mathrm{Proj}\left( \mathrm{Sym}^\cdot F^{\vee}\right)$ of $F$ parameterizes one dimensional subspaces in fibers of $F$. For a variety $X$, $D^b(X)$ denotes the bounded derived category of coherent sheaves on $X$. The zero locus of a section $s:\cO_X\to F$ of a vector bundle $F$ over a scheme $X$ is the closed subscheme of $X$ whose ideal is the image of $s^\vee:F^\vee\to \cO_X$.

\bigskip

\section{Preliminaries}\label{s2}

In this section we collect basic definitions and facts that we will use.

\subsection{Semiorthogonal decomposition}\label{s2.1}

We review semiorthogonal decompositions of derived categories of coherent sheaves. Our basic reference is \cite{Kuz3}.

\begin{defi}\label{d2.1}
Let $\mathcal{T}$ be a triangulated category. A \emph{semiorthogonal decomposition} of $\mathcal{T}$ is a sequence of full triangulated subcategories $ \mathcal{A}_1, \cdots, \mathcal{A}_n $ satisfying the following properties:
\\
(1) $Hom_{\mathcal{T}}(a_i, a_j)=0$ for any $a_i \in \mathcal{A}_i, a_j \in \mathcal{A}_j$ with $i > j$; \\
(2) the smallest triangulated subcategory of $\mathcal{T}$ containing $ \mathcal{A}_1, \cdots, \mathcal{A}_n $ is $\mathcal{T}.$

We will write $\mathcal{T} = \langle \mathcal{A}_1, \cdots, \mathcal{A}_n \rangle$ to denote the semiorthogonal decomposition.
\end{defi}

The notion of semiorthogonal decomposition plays a key role in the study of derived categories of algebraic varieties and lots of semiorthogonal decompositions have been constructed. We recall several constructions of semiorthogonal decompositions that we will use later.

Orlov proved that derived categories of projective bundles have semiorthogonal decompositions. 

\begin{theo}\label{t2.2} \cite{Orlov}
Let $E$ be a vector bundle of rank $k+1$ on $X$ and $ \pi : X=\PP E^\vee \to Y $ be a projective bundle. Then 
$$ \Phi_i : D^b(Y) \to D^b(X), ~~~~~ \mathcal{F} \mapsto L\pi^*(\mathcal{F}) \otimes^L \mathcal{O}_X(i) $$ 
is a fully faithful functor for each $i \in \ZZ$ and we have the semiorthogonal decomposition
$$ D^b(X)=\langle \Phi_0(D^b(Y)),\Phi_1(D^b(Y)),\cdots,\Phi_k(D^b(Y)) \rangle. $$
\end{theo}

Using this semiorthogonal decomposition, Orlov obtained the blowup formula.

\begin{theo}\label{t2.3} \cite{Orlov}
Let $Y \subset S$ be a smooth subvariety which is a local complete intersection of codimension $c$ in $S$. Let $\pi : X \to S$ be the blowup along $Y$ and $j : D \to X$ be the exceptional divisor. Then
$$ L\pi^* : D^b(S) \to D^b(X), $$
$$ \Phi_i : D^b(Y) \to D^b(X), ~~~~~ \mathcal{F} \mapsto Rj_*(L\pi^*(\mathcal{F}) \otimes^L \mathcal{O}(i)) $$ 
are fully faithful functors for all $i \in \ZZ$ and we have the semiorthogonal decomposition
$$ D^b(X) = \langle L\pi^*(D^b(S)), \Phi_0(D^b(Y)),\Phi_1(D^b(Y)),\cdots,\Phi_{c-2}(D^b(Y)) \rangle .$$
\end{theo}

%Projective bundles are special cases of Brauer-Severi varieties. 

%\begin{defi}\label{d2.4} %(Brauer-Severi variety)
%A \emph{Brauer-Severi variety} is a variety $X$ together with a morphism $\pi : X \to Y$ such that all scheme theoretic fibers of $\pi$ are isomorphic to $\PP^k$ for a fixed $k$.
%\end{defi}

%Bernardara proved that derived categories of Brauer-Severi varieties have semiorthogonal decompositions.

%\begin{theo}\label{t2.5} \cite{Bern}
%Let $ \pi : X \to Y $ be a Brauer-Severi variety and $\beta \in Br(Y)$ be its Brauer class. There are fully faithful functors
%$$ \Phi_i : D^b(Y,\beta^i) \to D^b(X), ~~~~~ \mathcal{F} \mapsto L\pi^*(\mathcal{F}) \otimes^L \mathcal{O}(i) $$ 
%for each $i \in \ZZ$ where $D^b(Y,\beta^i)$ is the derived category of $\beta^i$-twisted sheaves on $Y$ and we have the semiorthogonal decomposition
%$$ D^b(X)=\langle \Phi_0(D^b(Y)),\Phi_1(D^b(Y,\beta)),\cdots,\Phi_k(D^b(Y,{\beta}^k)) \rangle .$$
%\end{theo}

We will use the above results to prove that complete intersection Calabi-Yau varieties of codimension at most 2 are Fano visitors. 

For the proof of our main theorem (Theorem \ref{t1.1}), we will use the following theorem of Orlov (cf. \cite{Orlov2}) about hyperplane fibrations which was generalized by Ballard, Deliu, Favero, Isik and Katzarkov recently for higher degree hypersurface fibrations (cf. \cite[Theorem 3.2 \& Corollary 3.4]{BDFIK}).
% described derived categories of hypersurface fibrations in \cite{BDFIK} and their work will be essential in our proof.

\begin{theo}\label{t2.6}\cite[Proposition 2.10]{Orlov2}
Let $E$ be a vector bundle of rank $r\ge 2$ over a smooth projective variety $S$
and let $Y=s^{-1}(0)\subset S$ denote the zero locus of a regular section $s \in H^0(S,E)$ such that $ \dim Y = \dim S - \mathrm{rank}\, E$. 
Let $X=w^{-1}(0) \subset \PP E^\vee$ be the zero locus of the section $w\in H^0(\PP E^\vee, \cO_{\PP E^\vee}(1))$  
determined by $s$ under the natural isomorphism
$H^0(S,E)\cong H^0(\PP E^\vee, \cO_{\PP E^\vee}(1))$. 
Then we have the semiorthogonal decomposition
$$ D^b(X)= \langle q^*D^b(S), \cdots, q^*D^b(S) \otimes_{\cO_X} {\cO_X}(r-2), D^b(Y) \rangle .$$
\end{theo}

\subsection{Fano varieties}\label{s2.2}
In this subsection, we recall some properties of Fano varieties.
\begin{defi}\label{d2.7}
A smooth projective variety $X$ is called \emph{Fano} if its anticanonical line bundle $K_X$ is ample, i.e. the dual $K_X^\vee$ of $K_X$ is ample.
\end{defi}

Fano varieties have many nice properties.

\begin{theo}\label{t2.8}\cite[Theorem 2.1]{FH}
For any positive integer $n$, there are only finitely many deformation equivalence classes of Fano varieties of dimension $n$.
\end{theo}

\begin{exam}\label{e2.9}
A two dimensional Fano variety is called a \emph{del Pezzo surface}. It is well known that there are 10 deformation equivalence classes of del Pezzo surfaces.
\end{exam}

\begin{theo}\label{t2.10}\cite[Theorem 2.2]{FH}
Fano varieties are rationally connected.
\end{theo}

\begin{theo}\label{t2.12}\cite[Theorem 2.3]{FH}
The Mori cone of a Fano variety is a rational polyhedral cone generated by classes of rational curves.
\end{theo}

Finally we recall the definition of Fano visitor.
\begin{defi}\label{d2.12}\cite[Definition 2.11]{BBF}
A smooth projective variety $Y$ is a \emph{Fano visitor} if there exists a Fano variety $X$ and a fully faithful functor $D^b(Y) \to D^b(X)$ such that $D^b(X)=\langle ^{\perp}D^b(Y), D^b(Y) \rangle.$ 

Such an $X$ is called a \emph{Fano host} of $Y$.
\end{defi}

\subsection{Ample and nef vector bundles}\label{s2.3}

We recall the notions of ample and nef vector bundles.

\begin{defi}\label{d2.13}
A vector bundle $E$ on $X$ is \emph{ample (resp. nef)} if the Serre line bundle $\mathcal{O}_{\PP E^\vee}(1)$ is an ample (resp. nef) line bundle on $\PP E^\vee$.  
\end{defi}

We will use the following property of ample (resp. nef) bundles.

\begin{theo}\label{t2.14}\cite[Proposition 6.1.13 \& Theorem 6.2.12]{Laz2}
Let $X$ be a projective variety. Then direct sums and extensions of ample (resp. nef) bundles are ample (resp. nef).
\end{theo}

Let us recall the well known ampleness criterion of Kleiman's.
\begin{theo}\cite[Theorem 1.27]{Deb} Let $X$ be a projective variety and $D$ be a Cartier divisor. Then $D$ is ample if and only if $C \cdot D > 0$ for all nonzero $C \in \overline{NE}(X)$.
\end{theo}
For line bundles on toric varieties, the criterion is much simpler to check.
\begin{theo} \label{t2.15} \cite[Theorem 6.3.13]{CLS} (Toric Kleiman Criterion)\\ 
Let $D$ be a Cartier divisor on a complete toric variety $X_{\Sigma}$. Then $D$ is ample if and only if $D \cdot C > 0$ for all torus-invariant irreducible curves $C \subset X_{\Sigma}$.
\end{theo}

\bigskip

\section{Complete intersection Calabi-Yau varieties}\label{s3}

In this section we prove that every complete intersection Calabi-Yau variety is a Fano visitor. 
%Indeed we can go to the proof of the main theorem directly. However we treat Calabi-Yau case separately since we can approach this case in several ways and the dimensions of related Fano varieties are relatively small.

\subsection{Construction}\label{s3.1}
%Kuznetsov and others discovered semiorthogonal decompositions of derived categories of Fano varieties which contain those of Calabi-Yau varieties (references?). 

%Segal and Thomas reversed the problem and asked if one can find a Fano $Y_1$ for a Calabi-Yau $Y_2$ such that $D(Y_2)\subset D(Y_1)$. They proved that for a general quintic 3-fold $Y_2$, there is a Fano 11-fold $Y_1$ such that there is a full and faithful embedding $$D(Y_2)\hookrightarrow D(Y_1).$$ 

\def\Sym{\mathrm{Sym} }
\def\Pic{\mathrm{Pic} }
\def\txi{\tilde{\xi} }
%Ballard et al considered the following situation: Let $S$ be a smooth projective variety and $E$ is a vector bundle of rank $r$ and let $d<r$ be a natural number. Let $$q:\PP E^\vee\to S$$ denote the projectivization of the bundle $E\to S$. In this paper we always use the subspace convention for projectivizations so that a point over $s\in S$ is a one dimensional subspace of $E|_s$. Let $\xi\in H^0(S,\Sym^d(E)\otimes L)$ for $L\in \Pic(S)$. By the isomorphism
%$$H^0(S,\Sym^d(E)\otimes L)\cong H^0(\PP E^\vee, \cO_{\PP E^\vee}(d)\otimes q^*L),$$
%we obtain a section $\txi$ of the line bundle $\cO_{\PP E^\vee}(d)\otimes q^*L$. Let $X=\txi^{-1}(0)$.  
%Then a theorem of Ballard et al says that the derived category $D(X)$ admits a semiorthogonal decomposition whose factors are $q^*D(S), \cdots, q^*D(S)\otimes_{\sO_X}\sO_X(r-d-1)$ and some matrix factorization category. 

%If $E$ is sufficiently ample so that $Y:=\xi^{-1}(0)$ is nonsingular of codimension $\binom{d+r-1}{d}$, then 
%$$K_Y=K_S\otimes (\Sym^dE\otimes L)|_Y,$$
%$$K_X=q^*(L\otimes K_S\otimes\det E)\otimes \sO_{\PP E^\vee}(d-r)|_X.$$

%\bigskip
%\vspace{3cm}

We assume that $E$ is a vector bundle of rank $r \geq 2$ over a smooth projective $S$ and that we have a section $s\in H^0(S,E)$ that defines a complete intersection nonsingular Calabi-Yau subvariety $Y=s^{-1}(0)$ of codimension $r$ in $S$. By the isomorphism 
$$H^0(S,E)\cong H^0(\PP E^\vee, \cO_{\PP E^\vee}(1)),$$
we have a section $w$ of $\cO_{\PP E^\vee}(1)$ whose zero locus $X=w^{-1}(0)$ is a hypersurface in $\PP E^\vee$.\footnote{This construction has been known and used before. Richard Thomas kindly informed us that it appeared in \cite{IM} as Cayley's trick. Atanas Iliev told us that he learned it from \cite{Cox, Nagel}.  Later we also found the same construction in \cite[\S2]{Orlov2}.}
It is straightforward to check that $X$ is smooth by using local coordinates.
More explicitly, $w$ is the composition
$$\cO_{\PP E^\vee}\mapright{q^*s} q^*E\lra \cO_{\PP E^\vee}(1)$$
of $q^*s$ and the dual of the universal family $\cO_{\PP E^\vee}(-1)\to q^*E^\vee$ over $\PP E^\vee$. 
Here $q:\PP E^\vee\to S$ denotes the projective bundle map.
Note that $$\dim X=\dim S+r-2 \and \dim Y=\dim S-r.$$

\begin{lemm}\label{l3.1}
$X$ is Fano if $(K_S\otimes \det E)^{-1}$ is nef and $E$ is ample.
\end{lemm}
\begin{proof} Obviously $K_{\PP E^\vee}=q^*K_S\otimes K_{\PP E^\vee/S}$.
From the exact sequence 
$$0\lra \cO_{\PP E^\vee}\lra q^*E^\vee\otimes \cO_{\PP E^\vee}(1) \lra T_{\PP E^\vee/S}\lra 0,$$
we have $K_{\PP E^\vee/S}=(q^*\det E)\otimes \cO(-r)$. Hence, 
$$K_X=K_{\PP E^\vee}\otimes \cO(1)|_X\cong q^*(K_S\otimes \det E)|_X\otimes \cO_X(1-r).$$
Since $E$ is ample, $\cO_X(1)$ is ample and so is $K_X^{-1}$.
\end{proof}

\subsection{Zero sections of rank 2 vector bundles}\label{s3.2}

In this subsection, we consider a special case of $r=2$ in the above construction. Note that $\dim X=\dim S$ and $\dim Y=\dim S-2$.

\def\Proj{\mathrm{Proj} }
\begin{lemm}\label{l3.2}
The composition $\phi:X\hookrightarrow \PP E^\vee \mapright{q} S$ is the blowup along $Y$. 
\end{lemm} 
\begin{proof} Since $Y=s^{-1}(0)$, the ideal sheaf $I$ of $Y$ in $S$ is the image of the dual $s^\vee:E^\vee\to \cO_S$ of $s$. The surjection $E^\vee\to I$ then induces a surjection $\Sym(E^\vee)$ of the symmetric algebra onto $\oplus_{n\ge 0} I^n$ whose $\Proj$ is by definition the blowup $\tilde{S}$ of $S$ along $Y$. Hence we obtain an embedding of $\tilde{S}$ into $\Proj (\Sym E^\vee)=\PP E$.

Since the rank of $E$ is $2$, we have an isomorphism $\det E\otimes E^\vee\cong E$ which induces an isomorphism
$$\PP E^\vee\cong \PP (E^\vee\otimes \det E)\cong \PP E.$$ 
We thus have an embedding  $\tilde{S}\hookrightarrow \PP E\cong \PP E^\vee$. By direct local calculation, one finds that this is precisely the zero locus of $w$. The completes the proof.
\end{proof}

\begin{prop}\label{p3.3}
Let $Y$ be a complete intersection Calabi-Yau variety of codimension 1 or 2. Then $Y$ is a Fano visitor.
\end{prop}
\begin{proof}
Let $Y$ be a smooth Calabi-Yau hypersurface in $\PP^n$. The section $s_0$ defining the embedding $\PP^n\to \PP^{n+1}$ and the defining section $s_1\in H^0(\PP^n, \cO_{\PP^n}(n+1))\subset H^0(\PP^{n+1}, \cO_{\PP^{n+1}}(n+1))$ of $Y$ give us a section $s$ of $E=\mathcal{O}_{\PP^{n+1}}(1) \oplus \mathcal{O}_{\PP^{n+1}}(n+1)$ over $S=\PP^{n+1}$, whose zero locus is $Y$. By Lemma \ref{l3.2}, $X \subset \PP E$ is the blowup of $S$ along $Y$. By Theorem \ref{t2.3}, we have the semiorthogonal decomposition
$$ D^b(X) = \langle D^b(S),D^b(Y) \rangle .$$ 
By Lemma \ref{l3.1}, $X$ is Fano since $E$ is ample by Theorem \ref{t2.14} and $K_S\otimes \det E=\cO_{\PP^{n+1}}(-n-2)\otimes \cO_{\PP^{n+1}}(n+2)\cong \cO_{\PP^{n+1}}$ is trivial. Therefore $Y$ is a Fano visitor. 

When $Y$ is a complete intersection Calabi-Yau variety of codimension 2 in $S=\PP^n$, $Y$ is the zero locus of a section of $E=\cO_{\PP^n}(a)\oplus \cO_{\PP^n}(b)$ for $a,b>0$ with $a+b=n+1$. By the same argument as in the codimension 1 case, $Y$ is  a Fano visitor.
\end{proof}

\subsection{General complete intersection Calabi-Yau varieties}
For general complete intersection Calabi-Yau varieties, we use Theorem \ref{t2.6}.
\begin{prop}\label{p3.4} 
A complete intersection Calabi-Yau varieties $Y\subset \PP^n$ of codimension greater than 1 is a Fano visitor.
\end{prop}
\begin{proof}
Let $Y$ be a complete intersection Calabi-Yau variety in $S=\PP^n$. Then $Y$ is the zero locus of a section of an ample vector bundle $E=\oplus_{i=1}^r \cO_{\PP^n}(a_i)$ for $a_i>0$, $\sum a_i=n+1$. Since $K_{\PP^n}\otimes \det E$ is trivial and $E$ is ample, $X$ is Fano by Lemma \ref{l3.1}. By Theorem \ref{t2.6}, we have the semiorthogonal decomposition
$$ D^b(X)= \langle q^*D^b(S), \cdots, q^*D^b(S) \otimes_{\cO_X} {\cO_X}(r-2), D^b(Y)\rangle .$$
Therefore $Y$ is a Fano visitor.
\end{proof}

Combining Propositions \ref{p3.3} and \ref{p3.4}, we obtain the following.
\begin{theo}\label{t3.5}
All complete intersection Calabi-Yau varieties are Fano visitors.
More precisely, if $Y$ is a complete intersection Calabi-Yau variety of codimension $r$ in a projective space $\PP^n$, then there is a Fano host $X$ of dimension $n+r-2$ for $r\ge 2$ and $n+1$ for $r=1$. Moreover the Fano host $X$ is a hypersurface of a projective bundle over a projective space.
\end{theo}

Genki Ouchi observed (\cite{Ouchi}) that by the above arguments we can actually choose a lower dimensional Fano host for a \emph{general} complete intersection Calabi-Yau varieties.
\begin{prop}\label{p3.6} \cite{Ouchi}  Let $Y\subset \PP^{d+r}$ be a complete intersection Calabi-Yau variety of dimension $d$ defined by the vanishing of homogeneous polynomials $f_1,\cdots, f_r$. Suppose $r\le 2$ or $Y$ is general in the sense that we can choose the defining equations such that the projective variety $S$ defined by the vanishing of $f_3, \cdots, f_r$ is smooth. Then  there is a Fano host $X$ of $Y$ of dimension $d+2$.
\end{prop}
\begin{proof}
When $r\le 2$, the proposition follows from Proposition \ref{p3.3}. When $r\ge 3$, it is easy to see that $S$ is Fano because $Y$ is Calabi-Yau. Let $a_1, a_2>0$ denote the degrees of $f_1, f_2$ and let $E=\cO_{\PP^{d+r}}(a_1)|_S\oplus \cO_{\PP^{d+r}}(a_2)|_S$ which has a regular section $s$ defined by $(f_1,f_2)$. By definition, $s^{-1}(0)=Y$ and we have a smooth variety $X=w^{-1}(0)\subset \PP E^\vee$ by the construction in \S\ref{s3.1} which is Fano by Lemma \ref{l3.1}. By Lemma \ref{l3.2}, we find that $X$ is the blowup of $S$ along $Y$ and hence there is a fully faithful embedding $D^b(Y)\to D^b(X)$ by Theorem \ref{t2.3}. So we proved the proposition.
\end{proof}

%We will show in \cite{KL} that under the assumptions of the proposition, $d+2$ is indeed the smallest possible dimension of Fano hosts. 
%\red At least, make sure to check the cases of K3 and CY3.
%Question: Does the construction work for general type case? \black

\bigskip

\section{General complete intersection varieties}\label{s4}

The goal of this section is to prove the following.
\begin{theo}\label{t4.3} (Main Theorem)\\
Every smooth complete intersection is a Fano visitor. Moreover, we can choose a Fano host which is a hypersurface of a projective bundle over a projective space.
\end{theo}

Let $S$ be a smooth projective variety and $s$ be a regular section of a vector bundle $E$ of rank $r$ over $S$ whose zero locus is a smooth subvariety $Y=s^{-1}(0)$ of codimension $r$. By the isomorphism 
$$H^0(S,E)\cong H^0(\PP E^\vee, \cO_{\PP E^\vee}(1)),$$
we have a section $w$ of $\cO_{\PP E^\vee}(1)$ whose zero locus $X=w^{-1}(0)$ is a hypersurface in $\PP E^\vee$ which is smooth by direct local calculation. 

The key point of the proof of Proposition \ref{p3.4} is recapitulated as follows.
\begin{prop}\label{p4.1}
If $X$ is Fano then $Y$ is a Fano visitor.
\end{prop}
\begin{proof}
By Theorem \ref{t2.6}, we have the semiorthogonal decomposition
$$ D^b(X)= \langle q^*D^b(S), \cdots, q^*D^b(S) \otimes_{\cO_X} {\cO_X}(r-2), D^b(Y) \rangle .$$
Therefore $Y$ is a Fano visitor.
\end{proof}

Thus, to prove that a variety $Y$ is a Fano visitor, we only have to find an embedding $Y\hookrightarrow S$ such that $Y=s^{-1}(0)$ for a regular section $s$ of a vector bundle $E$ over $S$ of rank $r=\mathrm{codim}_SY$ and that $X=w^{-1}(0)$ is Fano.
%Using this proposition we can prove that every smooth complete intersection variety is a Fano visitor.

\subsection{Construction}\label{s4.1}

Let $Y$ be a smooth complete intersection variety in $\PP^{n}$ of codimension $c$. We may assume that $Y$ is the zero locus of a regular section of $$\mathcal{O}_{\PP^n}(d_1) \oplus \cdots \oplus \mathcal{O}_{\PP^n}(d_c),\quad d_1 \geq \cdots \geq d_c \geq 1.$$ Let us fix a positive integer $r$ which is greater than $d_1+\cdots+d_c-n-c$ and $1-c$. 
Let $S=\PP^{n+r}$ and 
$$E=\mathcal{O}_S(1) \oplus \cdots \oplus \mathcal{O}_S(1) \oplus \mathcal{O}_S(d_1) \oplus \cdots \oplus \mathcal{O}_S(d_c)$$ be the rank $r+c$ ample vector bundle on $S$. After choosing $r$ sections of $\cO_S(1)$ that define the embedding $\PP^n\subset \PP^{n+r}$, we find that $Y$ is the zero locus $s^{-1}(0)$ of a section $s$ of $E$ and that $\mathrm{codim}_SY=r+c=\mathrm{rank}\, E$. By the isomorphism 
$$H^0(S,E)\cong H^0(\PP E^\vee, \cO_{\PP E^\vee}(1)),$$
we have a section $w$ of $\cO_{\PP E^\vee}(1)$ whose zero locus $X=w^{-1}(0)$ is a smooth hypersurface in $\PP E^\vee$.

\subsection{Proof of Theorem \ref{t4.3}}\label{s4.2}

By Proposition \ref{p4.1}, it suffices to show that $X$ is a smooth Fano variety.
%\begin{lemm}\label{t4.2}$X$ is a smooth Fano variety.\end{lemm}\begin{proof}

Because $Y$ is a smooth complete intersection, it is straightforward to check that $X$ is smooth by local calculation.
Let 
%$E=\mathcal{O}_S(1) \oplus \cdots \oplus \mathcal{O}_S(1) \oplus \mathcal{O}_S(d_1) \oplus \cdots \oplus \mathcal{O}_S(d_c)$ and 
$F=E \otimes \mathcal{O}_S(-1)= \mathcal{O}_S \oplus \cdots \oplus \mathcal{O}_S \oplus \mathcal{O}_S(d_1-1) \oplus \cdots \oplus \mathcal{O}_S(d_c-1).$ Then $\PP E^\vee$ is canonically isomorphic to $\PP F^\vee$ and $\mathcal{O}_{\PP E^\vee}(1)=\mathcal{O}_{\PP F^\vee}(1) \otimes q^*\mathcal{O}_S(1)$ where
$q:\PP E^\vee\to S$ denotes the projective bundle map. Consider the Euler sequence
$$ 0 \to \mathcal{O}_{\PP E^\vee} \to q^*E^{\vee} \otimes \mathcal{O}_{\PP E^\vee}(1) \to T_{\PP E^\vee/S} \to 0 $$
which implies
$$ K_{\PP E^\vee}^\vee \cong q^*K_S^\vee \otimes q^* \det E^{\vee} \otimes \mathcal{O}_{\PP E^\vee}(r+c) .$$
By the adjuction formula, 
$$ K^\vee_X = K^\vee_{\PP E^\vee} \otimes \mathcal{O}_{\PP E^\vee}(-1)|_X $$
$$ = q^*\mathcal{O}_S(n+1-d) \otimes \mathcal{O}_{\PP E^\vee}(r+c-1)|_X $$
$$ = q^*\mathcal{O}_S(n+r+c-d) \otimes \mathcal{O}_{\PP F^\vee}(r+c-1)|_X $$
where $d=d_1+\cdots +d_c$.
Since $r > d-n-c$ and $r>1-c$ by choice, $q^*\mathcal{O}_S(n+r+c-d)$ is a nef line bundle and $\mathcal{O}_{\PP F^\vee}(r+c-1)$ is also a nef line bundle on $\PP E^\vee=\PP F^\vee$ by Theorem \ref{t2.14} because $F$ is a direct sum of nef line bundles. 

Let $C$ be an irreducible curve in $\PP E^\vee=\PP F^\vee$. If $q(C)$ is a point, then degree of $\mathcal{O}_{\PP F^\vee}(r+c-1)|_C$ is positive because $\cO_{\PP F^\vee}(1)$ is ample on each fiber of $q:\PP F^\vee\to S$. If $q(C)$ is a curve, then the degree of $q^*\mathcal{O}_S(n+r+c-d)|_C$ is positive. Therefore for any irreducible curve $C \subset \PP F^\vee $, $q^*\mathcal{O}_S(n+r+c-d) \otimes \mathcal{O}_{\PP F^\vee}(r+c-1)|_C$ has positive degree. By toric Kleiman's criterion (Theorem \ref{t2.15}) or by the fact that the Picard number of $\PP F^\vee$ is 2, we find that $q^*\mathcal{O}_S(n+r+c-d) \otimes \mathcal{O}_{\PP F^\vee}(r+c-1)$ is an ample line bundle on the toric variety $\PP F^\vee$. Therefore its restriction $K_X^\vee = q^*\mathcal{O}_S(n+r+c-d) \otimes \mathcal{O}_{\PP F^\vee}(r+c-1)|_X $ is ample as desired.

\begin{rema}
Our construction and result work for all smooth projective varieties which are zero loci $Y=s^{-1}(0)$ of regular sections $s$ of
vector bundles $E$ over smooth projective varieties $S$ whose projectivizations $\PP E^\vee$ are Fano such that 
$\dim Y=\dim S-\mathrm{rank}\, E$ and that the hypersurfaces  $X=w^{-1}(0)$ are smooth Fano where $w$ is the section of $\cO_{\PP E^\vee}(1)$ given by $s$ via the isomorphism $H^0(S, E)\cong H^0(\PP E^\vee, \cO_{\PP E^\vee}(1))$. For instance, rank 2 Fano bundles over homogeneous varieties (cf. \cite{APJW, MOS}) may provide us with a sequence of Fano visitors.
\end{rema}

\bigskip

\section{Questions and problems}\label{s5}

As discussed in \S\ref{s1}, we can raise many interesting questions and problems related to Fano visitors and hosts.
Since we have sufficiently many Fano visitors, it seems reasonable to
introduce the following.
\begin{defi} %(Definition \ref{d5.1})
The \emph{Fano dimension} of a smooth projective variety $Y$ is the minimum dimension of Fano hosts $X$ of $Y$.
%of $\dim X$ for smooth projective 
%Fano varieties $X$ which admit fully faithful embeddings $D^b(Y)\hookrightarrow D^b(X)$. 
We define the Fano dimension to be infinite if no Fano hosts exist.

The \emph{Fano number} of a smooth projective variety $Y$ is the number of deformation equivalence classes of irreducible Fano hosts $X$ of $Y$ of minimal dimension (i.e. $\dim X$ equals the Fano dimension of $Y$).

Let $\kappa_Y(n)$ be the number of deformation equivalence classes of irreducible Fano hosts $X$ of $Y$ of dimension $n$. The function $\kappa_Y:\ZZ_{>0}\to \ZZ_{\ge 0}$ defined by $n\mapsto \kappa_Y(n)$ is called the \emph{Fano function} of $Y$. 
\end{defi}
The definitions of Fano number and Fano function make sense because of Theorem \ref{t2.8}. 
The following problem seems quite natural.

\begin{prob}
Find the Fano number of a quartic surface, a quintic 3-fold, or more generally a complete intersection Calabi-Yau variety.
\end{prob}

A direct consequence of Theorem \ref{t3.5} is that the Fano dimension of an elliptic curve is at most 3 and that of a quartic surface is at most 4 while that for a quintic 3-fold is at most 5. A Calabi-Yau hypersurface $Y\subset \PP^n$
has Fano dimension at most $n+1$. If $Y\subset\PP^n$ is a codimension $r$ Calabi-Yau complete intersection, then $Y$ has Fano dimension at most $n+r-2$.

\begin{exam}
The Fano dimension of an elliptic curve is 3 because the $K$-groups of Fano surfaces are finitely generated abelian groups while those of elliptic curves are not finitely generated.
\end{exam}
%Genki Ouchi pointed  out to us (cf. \cite{Ouchi}) that we can prove that the Fano dimensions of $m$-dimensional complete intersection Calabi-Yau varieties are at most $m+2$ using \ref{l3.1} and \ref{l3.2}. Recently we proved that the Fano dimension of a Calabi-Yau complete intersections of dimension $m$ is exactly $m+2$ in \cite{KL}.

%When $Y$ is a point, $\kappa_Y(1)=1$ and $\kappa_Y(2)=10$. This characterizes $Y=pt$. 
When $Y$ is $\PP^1$, $\kappa_Y(1)=1$. Conversely, if $Y$ is a smooth projective variety of dimension $>0$ with $\kappa_Y(1)=1$, then $Y$ is $\PP^1$ because $\kappa_Y(1)=1$ implies that $Y$ is either a point or $\PP^1$. Hence the Fano function characterizes $\PP^1$.
It will be interesting to investigate how strong the Fano function is as an invariant of the variety $Y$. 
\begin{conj}
The Fano function determines del Pezzo surfaces. Namely if a smooth projective variety $Y$ satisfies $\kappa_Y=\kappa_S$ for a del Pezzo surface $S$, then $Y\cong S$.
\end{conj}

\medskip

We may refine Bondal's question (Question \ref{q1.0}).
\begin{ques} \emph{(Refined Fano visitor problem)}\\
Let (P) be a property (for example, toric hypersurface or hypersurface or $h^{3,1}=1$) of a Fano variety. 
Let $Y$ be a smooth projective variety. Does there exist a Fano host $X$ of $Y$ with property (P)?
\end{ques}

For instance, we can raise the following question.
\begin{ques}
Let $Y$ be a K3 surface or more generally a holomorphic symplectic variety. Does there exist a Fano host $X$ with $h^{3,1}(X)=1$?
\end{ques}
For instance, Kuznetsov's result in \cite{Kuz1} says that certain cubic 4-folds are Fano hosts, with $h^{3,1}=1$, of K3 surfaces.

When (P) requires $X$ to be a hypersurface in a toric variety, Theorem \ref{t4.3} says that the answer is yes for all complete intersections. 

We can likewise refine the definitions of Fano dimension, Fano number and Fano function with property (P). If we let (P) require $X$ to be a hypersurface in a toric variety and let $\tau_Y(n)$ be the number of deformation equivalence classes of irreducible Fano hosts $X$ of dimension $n$, which is a hypersurface in a toric projective variety, then $\tau_Y:\ZZ_{>0}\to \ZZ_{\ge 0}$, $n\mapsto \tau_Y(n)$ is called the \emph{toric Fano function} of $Y$.

Since toric Fano varieties are well understood at least in low dimension, it may be possible to calculate the toric Fano function. 
\begin{prob}
Calculate the toric Fano functions for low dimensional varieties.
\end{prob}

\medskip

We may consider the Fano dimension as a function $\mathrm{fdim}:\Delta\to \ZZ_{> 0}$ when we are given a family of smooth projective varieties $\fX\to \Delta$ by considering the Fano dimensions of the fibers. Then it gives a stratification of the family and we may ask how nice this stratification is.
\begin{ques}
(1) Let $Y_t$, $t\in \Delta$ be a flat family of smooth projective varieties and let $X_0$ be a Fano variety such that $D^b(X_0)$ contains $D^b(Y_0)$. Under which condition can we find a flat family of Fano varieties $X_t$ such that $D^b(X_t)$ contains $D^b(Y_t)$ for each $t$? \\
(2) Is the function of Fano dimensions semicontinuous?
\end{ques}

\medskip

Often checking that a given variety is Fano can be difficult. So it may be helpful to consider a larger class of hosts which still share some of the nice properties of Fano. 
\begin{defi}
A smooth projective variety is called \emph{weak Fano} if its anticanonical line bundle is nef and big.
\end{defi}

The following is an immediate consequence of Kawamata-Viehweg vanishing theorem which tells us that weak Fano can be quite useful.
\begin{prop}
Let $X$ be a weak Fano variety. Then $h^i(X,\mathcal{O}_X)=0$ for all $i > 0$. Therefore every line bundle on $X$ is an exceptional object in $D^b(X)$.
\end{prop}
Consequently, derived categories of weak Fano varieties always have canonical semiorthogonal decompositions. So it is also reasonable to ask which triangulated categories are contained in the derived categories of weak Fano varieties. 
\begin{ques}\label{q5.10} \emph{(Weak Fano visitor problem)}\\
(1) Let $Y$ be a smooth projective variety. Is there a weak Fano variety $X$ such that $D^b(X)$ contains $D^b(Y)$? \\
(2) Let $\mathcal{T}$ be a triangulated category such as the derived category of a noncommutative scheme (cf. \cite{Kuz1}). Is there a weak Fano variety $X$ such that $D^b(X)$ contains $\mathcal{T}$?
\end{ques}

\bigskip

\bibliographystyle{amsplain}

\end{document}